\documentclass[12pt,a4paper]{amsart}
\usepackage{amsmath,amsfonts,amsthm,amsopn,color,amssymb,enumitem}
\usepackage[a4paper]{geometry}
\usepackage{palatino}
\usepackage{graphicx}
\usepackage[colorlinks=true]{hyperref}
\usepackage{cancel}
\hypersetup{urlcolor=blue, citecolor=red, linkcolor=blue}
\allowdisplaybreaks
\usepackage{cite}
\usepackage{relsize}
\usepackage{esint}
\usepackage{verbatim}
\usepackage{mathrsfs}
\usepackage{xcolor}
\usepackage{xfrac}

\newcommand{\e}{\varepsilon}

\newcommand{\R}{\mathbb{R}}

\newcommand{\Sph}{{\mathbb{S}}}

\newcommand{\Id}{\operatorname{Id}}

\newcommand{\grad}{\nabla}
\newcommand{\la}{\langle}
\newcommand{\ra}{\rangle}
\newcommand{\p}{\partial}
\newcommand{\Sphplus}{\Sph^{4}_{+}}

\newcommand{\de}{\partial}

\newtheorem{theorem}{Theorem}[section]
\newtheorem{proposition}[theorem]{Proposition}
\newtheorem{lemma}[theorem]{Lemma}
\newtheorem{corollary}[theorem]{Corollary}
\theoremstyle{definition}
\newtheorem{definition}[theorem]{Definition}

\title
[Kazdan--Warner obstructions for a 4$th$ order-problem]{Kazdan--Warner obstructions for a 4$th-$order boundary problem}

\author{Sergio Cruz-Bl\'azquez}
\address{Sergio Cruz Bl\'azquez, Departamento de Matemáticas, Universidad Autónoma de Madrid, Calle de Darwin 2, 28049 Madrid (Spain)}
\email{sergio.cruz@uam.es}
\author{Azahara DeLaTorre}
\address{Azahara DeLaTorre, Dipartimento di Matematica Guido Castelnuovo, Sapienza Universit\`a di Roma, Piazzale Aldo Moro 5, 00185 Roma (Italy)}
\email{azahara.delatorrepedraza@uniroma1.it}


\begin{document}
\maketitle

\begin{abstract}
We derive Kazdan--Warner type identities for the boundary problem of prescribing nonconstant interior $Q$ curvature and boundary $T$ curvature on the upper hemisphere $\Sphplus$ by a conformal change of the standard metric. Using the natural variational formulation and conformal variations generated by boundary-preserving conformal vector fields, we obtain nontrivial integral obstructions to solvability.
\end{abstract}

\section{Introduction} \label{s:intro}
Let $(\Sigma,g)$ be a compact surface and let $\tilde g=e^{2u}g$ be a conformal metric. The classical
\emph{Kazdan--Warner problem} asks which functions $K$ can arise as the Gaussian curvature of $\tilde g$.
Analytically, $u$ solves the Liouville-type equation
\begin{equation}\label{eq:intro-Liouville}
	-\Delta_g u + K_g = K\,e^{2u}\qquad \text{in }\Sigma,
\end{equation}
where $K_{g}$ denotes the Gaussian curvature of $\Sigma$ with respect to $g$. The Gauss--Bonnet theorem shows that the equation \eqref{eq:intro-Liouville} does not always admit a solution, since a topological constraint must hold, namely, $$\int_{\Sigma} K\,e^{2u}\,dV_g=2\pi\chi(\Sigma).$$
On the round sphere $(\mathbb S^2,g_0)$, this problem is known as the Nirenberg's problem, whose difficulty is intimately
related to the noncompactness of the group of conformal transformations of the sphere.
One important feature of this setting is the existence of nontrivial integral obstructions to solvability, as observed by Kazdan and Warner in \cite{KazdanWarner74,KazdanWarner75} by differentiating the Euler--Lagrange functional along conformal flows. A particularly interesting case arises when taking the gradients of the first spherical harmonics as conformal vector fields, as highlighted by Chang and Yang in \cite{ChangYangS2}. More precisely, if $u$ solves
$\Delta_{g_0}u+K\,e^{2u}=1$ on $\mathbb S^2$, then one has the explicit obstruction
\begin{equation*}
	\int_{\mathbb S^2}\big\langle \nabla_{g_0}K,\nabla_{g_0}x_j\big\rangle\,e^{2u}\,dV_{g_0}=0,
	\qquad j=1,2,3,
\end{equation*}
where $x_j$ are the coordinate functions of the embedding $\mathbb S^2\subset\R^3$. In particular, monotonicity of $K$ along a conformal direction prevents existence of solutions.

\medskip
A natural boundary analogue for \eqref{eq:intro-Liouville} is obtained by prescribing, besides the interior curvature, the geodesic curvature
of $\partial\Sigma$ to be equal to a given function $h$. If $\Sigma$ has boundary $\partial \Sigma$ with geodesic curvature $h_g$ and $\tilde g=e^{2u}g$, this problem reduces to solving the boundary problem
\begin{equation*}
		-\Delta_gu + K_g = K\,e^{2u} \quad \text{in }\Sigma,\qquad
		\frac{\partial u}{\partial \nu} + h_g = h\,e^{u} \quad \text{on }\partial\Sigma,
\end{equation*}
where $\nu$ denotes the external unit normal vector. The case of the upper hemisphere $\Sigma = \Sph^2_+$ with the standard round metric $g_0$ is especially delicate, again due to the noncompact action of its conformal group. In this case, we arrive at the following Liouville-type boundary problem
\begin{equation}\label{eq:disco}
	\left\{
	\begin{aligned}
		-\Delta u +1&= K\,e^{2u} && \text{in }\mathbb S^{2}_+,\\
		\frac{\partial u}{\partial \nu} &= h\,e^{u} && \text{on }\mathbb S^1,
	\end{aligned}
	\right.
\end{equation}
subject to the Gauss--Bonnet constraint 
\begin{equation*}
	\int_{\mathbb S^2_+} K\,e^{2u}\,dV_{g_0}
	\;+\;
	\int_{\mathbb S^1} h\,e^{u}\,ds_{g_0}
	\;=\;2\pi.
\end{equation*}

This problem with nonconstant $K$ and $h$ from the point of view of the disk has been studied in \cite{CBRuiz, ruiz24, struwe}, and more recently in \cite{LSRRS}.

\medskip
As in the case of the whole sphere, further obstructions to the existence of solutions in the form of Kazdan--Warner identities persist in the hemisphere. More precisely, Hamza (\cite{Hamza}) showed, via a careful integration by parts, that if $u$ solves \eqref{eq:disco}, then 
\begin{equation}\label{eq:HamzaKW}
	\int_{\mathbb S^2_+} e^{2u}\,\langle \nabla K,\nabla F\rangle\,dV_{g_0}
	\;+\;
	4\int_{\mathbb S^1} e^{u}\,\langle \nabla h,\nabla F\rangle\,ds_{g_0}
	\;=\;0 .
\end{equation}
for every $F$ in the span of the restrictions of the coordinate functions. In particular, \eqref{eq:HamzaKW} yields a simple nonexistence criterion: if there exists such an $F$
for which $\langle \nabla K,\nabla F\rangle$ and $\langle \nabla h,\nabla F\rangle$ have the same strict
sign, then the problem admits no solution.

\medskip
If $(M,g)$ is a closed Riemannian manifold of dimension four, Branson's $Q$-curvature and the conformally covariant Paneitz operator $P_g^4$ provide a natural analogue of the Gaussian curvature in the context of conformal geometry. Under the change of the metric $\tilde g=e^{2u}g$, they satisfy the transformation law
\begin{equation}\label{eq:Q}
P_g^4u+2Q_g = 2Q_{\tilde g}\,e^{4u}\qquad \text{in }M,
\end{equation}
see \cite{Branson,Paneitz}. In view of \eqref{eq:Q}, one could address the existence of conformal metrics $\tilde g$ with prescribed $Q$-curvature equal to a given function $Q$ (see \cite{brendle2003, djadli-malchiodi-Q1, ChangYangAnnals95}). In this regard,
the case of the sphere $\Sph^4$ is again particularly challenging and so far only the partial results obtained in \cite{wei-xu1998, malchiodi-struwe2006,LiLiLiu} for $Q>0$ are available. Notably, this case displays the same obstruction mechanism as in the two-dimensional Nirenberg problem. Indeed, Chang and Yang pointed out in \cite{ChangYangAnnals95} that, if $u$ solves
\[
P_{g_0}^4u+2Q_{g_0}=2Q\,e^{4u}\qquad \text{on }\mathbb S^4,
\]
then for the coordinate functions $x_j$ on $\mathbb S^4\subset\mathbb R^5$ one has
\begin{equation}\label{eq:KW-S4}
	\int_{\mathbb S^4}\big\langle \nabla_{g_0}Q,\nabla_{g_0}x_j\big\rangle\,e^{4u}\,dV_{g_0}=0,
	\qquad j=1,\dots,5,
\end{equation}
which again yields immediate nonexistence criteria under suitable monotonicity of $Q$ along the Euclidean directions.

\medskip
If $M$ has boundary, Chang and Qing (\cite{ChangQingI,ChangQingII}) introduced a third order boundary operator $P_g^3$
and the corresponding boundary curvature $T_g$, so that
\begin{equation}\label{eq:intro-QT}
	P_g^4u+2Q_g = 2Q_{\tilde g}\,e^{4u}\quad \text{in }M,
	\qquad
	P_g^3u+T_g = T_{\tilde g}\,e^{3u}\quad \text{on }\partial M.
\end{equation}
Together with the Gauss--Bonnet--Chern formula, \eqref{eq:intro-QT} suggests that the pair $(Q_g,T_g)$
is the natural four-dimensional counterpart of $(K_g,h_g)$ on surfaces.

\medskip
If one aims to prescribe $Q$ and $T$ curvatures on $M$ and $\de M$, one needs to solve \eqref{eq:intro-QT} with $Q_{\tilde g}=Q$ and $T_{\tilde g}=T$. For the sake of solvability, it is natural to impose an additional first order boundary condition. In order to maintain the conformal nature of the problem, one natural condition to work with is to impose that $(M,e^{2u}g)$ has minimal boundary, that is,
$$\frac{\partial u}{\partial \nu} + \frac{1}{3}H_gu = 0 \quad \mbox{on } \partial M,$$
where $\nu$ denotes the unit normal vector pointing outwards.

\smallskip

In general terms, this problem has not been widely addressed, and the available results cover mostly the particular cases in which both curvatures are constant, and one of them is identically equal to zero, see \cite{ChangQingII,ndiaye2008, ndiaye2009, ndiaye2011}. Moreover, all these results rule out the manifolds that are conformally equivalent to the hemisphere. The latter case, including nonconstant curvatures, was recently addressed in \cite{CruzDeLaTorrePreprint}.

\medskip
The present work is precisely concerned with the model case of the upper hemisphere
$\Sphplus\subset\mathbb R^5$ endowed with the standard metric $g$, for which $\partial\Sphplus=\mathbb S^3$ is totally geodesic. In this setting we have
$P_g^4=\Delta_g^2-2\Delta_g$ and $Q_g= 3$. Moreover, the boundary is totally geodesic, which means the second fundamental form $\mathbb{I}_g$ vanishes everywhere on $\Sph^3$. This gives $T_g= H_g = 0$ and, for functions satisfying
$\partial_\nu u=0$, one has $P_g^3u=-\frac12\,\partial_\nu(\Delta_g u)$.
Therefore, prescribing smooth functions $Q$ on $\Sphplus$ and $T$ on $\mathbb S^3$ as the $Q$ and $T$ curvatures
of $\tilde g=e^{2u}g$ leads to the boundary problem

\medskip
\begin{equation}\label{eq:P}\tag{\textcolor{red}{QT}}
	\left\lbrace
	\begin{array}{rll}
		\Delta_g^{2}u-2\Delta_g u+6 &= 2Q\,e^{4u} & \text{in }\Sphplus,\\[5pt]
		-\dfrac{\de \Delta_g u}{\de \nu} &= 2T\,e^{3u} & \text{on }\mathbb S^3,\\[8pt]
		\dfrac{\de u}{\de \nu} &=0 & \text{on }\mathbb S^3,
	\end{array}\right.
\end{equation}
where $\nu = -x_5$, together with the Gauss--Bonnet--Chern constraint
\begin{equation*}
	\int_{\Sphplus} Q\,e^{4u}\,dV_g+\int_{\mathbb S^3} T\,e^{3u}\,dS_g = 4\pi^2.
\end{equation*}
In \cite{CruzDeLaTorrePreprint} we initiated the study of \eqref{eq:P} from the variational point of view,
deriving existence results under symmetry assumptions by exploiting a mean-field type reformulation and
sharp higher order Moser--Trudinger inequalities on $\Sphplus$.

\medskip
This paper goes in the opposite direction: we investigate \emph{nonexistence phenomena} for \eqref{eq:P} by
deriving Kazdan--Warner type identities in the spirit of \eqref{eq:HamzaKW} and \eqref{eq:KW-S4}.
The proof follows the classical strategy of differentiating the energy along conformal orbits of $u$ generated by
boundary-preserving conformal vector fields, as done in \cite{ChangYangS2}.
More precisely, if $\{\phi_t\}$ is a conformal flow on $\Sphplus$ preserving the boundary and
$\phi_t^*g=e^{2P_t}g$ for some smooth function $P_t$ on $\Sph^4_+$, we test the Euler--Lagrange functional along the curve
\[
u_t=u\circ\phi_t+P_t.
\]
However, the fourth-order nature of \eqref{eq:P} requires establishing a cocycle identity for the quadratic part of the energy induced by $P^4_g$ and $P^3_g$. Differentiating at $t=0$ and using this identity together with conformal covariance yields an integral relation linking $X(Q)$ and $X(T)$, where $X$ is the infinitesimal generator of $\phi_t$. Our main result reads as follows:
\begin{theorem}\label{thm:KW}
	Let $u\in\mathcal{H}$ be a solution to \eqref{eq:P}. Then for every boundary-preserving conformal vector field $X$ on $\Sphplus$,
	\begin{equation}\label{eq:KW}
		\int_{\Sphplus}X(Q)e^{4u}\,dV_{g}+\frac{4}{3}\int_{\Sph^{3}}X(T)e^{3u}\,ds_{g}=0.
	\end{equation}
\end{theorem}
As a consequence, the problem admits no solutions if $Q$ and $T$ are simultaneously monotone along a conformal direction. In particular, we obtain the analogue of \eqref{eq:HamzaKW} by choosing $X$ among the standard conformal vector fields associated to the coordinate functions.

\medskip

The paper is organized as follows. In Section~\ref{sec:var} we recall the direct variational formulation of
\eqref{eq:P} and the basic conformal invariance properties needed in the argument.
Section~\ref{sec:KW} is devoted to the derivation of the Kazdan--Warner identity on $\Sphplus$.
Finally, in Section~\ref{sec:nonexist} we discuss several concrete nonexistence criteria and illustrative examples.

\section{The variational setting}\label{sec:var}

\medskip
Let
\begin{equation*}
\mathcal{H}(M)=\left\{u\in H^{2}(M):\ \frac{\partial u}{\partial \nu}=0 \text{ on }\de M\right\}.
\end{equation*}
We adopt the convention that, when we write $\mathcal H$, we are referring to $\mathcal H(\Sph^4_+)$.

\medskip
\begin{definition} We define the energy functional $I:\mathcal H\to \R$ by
	\begin{equation}\label{eq:energy}
		\begin{split}
			I(u)&=
			\int_{\Sphplus}(\Delta_{g}u)^{2}\,dV_{g}
			+2\int_{\Sphplus}|\grad u|^{2}\,dV_{g}
			+12\int_{\Sphplus}u\,dV_{g}
			\\&-\int_{\Sphplus}Qe^{4u}\,dV_{g}
			-\frac{4}{3}\int_{\Sph^{3}}Te^{3u}\,dS_{g}.
		\end{split}
	\end{equation}
\end{definition}
Integrating by parts and using the condition $\frac{\de u}{\de \nu} = 0$ on $\Sph^3,$ it is easy to check that critical points of $I$ solve \eqref{eq:P} in the usual weak sense. More precisely, if $u$ is a critical point of $I$ then
\begin{equation*}
\begin{split}
&\int_{\Sphplus}\Delta_{g}u\,\Delta_{g}v\,dV_{g}
+2\int_{\Sphplus}\la \grad u,\grad v\ra\,dV_{g}
+6\int_{\Sphplus}v\,dV_{g}\\
&-\int_{\Sph^{3}}2Te^{3u}v\,dS_{g}
-\int_{\Sphplus}2Qe^{4u}v\,dV_{g}=0
 \qquad \forall\,v\in\mathcal H,
\end{split}
\end{equation*}
which is the weak formulation of \eqref{eq:P}. The details of this computation can be found in Appendix \ref{app:wf}.

\bigskip
Now, we decompose our functional in three parts that will be studied separately. The following definitions and result do not require us to be in the particular setting of the hemisphere, so we will place ourselves in a more general context. 

\medskip

\begin{definition}\label{def:quad} We define the quadratic form $Q_g:\mathcal H(M)\times \mathcal H(M)\to \R$ as
\begin{equation*}
Q_g(u,v)=\int_M (P^4_gu)v\,dV_g+2\int_{\de M}(P^3_gu)v\,ds_g\quad u,v\in \mathcal H(M).
\end{equation*}
\end{definition}
 In the case of $(\mathbb S^4_+,g)$, we can integrate by parts twice, using the definitions of $P^4_g$ and $P^3_g$ introduced in Section \ref{s:intro}, to obtain the simplified expression:
\begin{equation}\label{Quad}
	\begin{split}
		Q_g(u,v)&=\int_{\Sphplus}(\Delta^2u-2\Delta u)v\,dV_g-\int_{\Sph^{3}}\frac{\partial \Delta u}{\partial \nu}v\,ds_g \\ & =\int_{\Sphplus}\Delta_gu\Delta_gv+2\int_{\Sph^3}\nabla_gu\cdot\nabla_gv,\quad u,v\in\mathcal H.
	\end{split}
\end{equation}

\begin{definition}\label{def:S} Let $(M,g)$ be a $4-$dimensional Riemannian manifold with boundary $\partial M$, and let $Q_g$ and $T_g$ be its Q and boundary T curvatures. We define the functional $S:\mathcal H(M)\to\R$ by
	$$S(u)=Q_g(u,u)+4\int_{\Sphplus}Q_gu\,dV_g+4\int_{\Sph^3}T_g u\,ds_g,\quad u\in \mathcal H(M).$$
\end{definition}
\begin{lemma}\label{lemma:conformal} Let $u\in \mathcal H(M)$ be a solution to \eqref{eq:intro-QT}, and call $\tilde g = e^{2u}g$. Then, for every $v\in\mathcal H(M)$, it holds:
	$$S_g(u+v)=S_g(u)+S_{\tilde g}(v).$$
\end{lemma}
\begin{proof}
Expanding by linearity,
\begin{align*}
S_g(u+v) &= Q_g(u+v,u+v)+4\int_{M}Q_g(u+v)\,dV_g+4\int_{\de M}T_g(u+v)\,ds_g \\ &= S_g(u)+Q_g(v,v)+2Q_g(u,v)+4\int_{M}Q_gv\,dV_g+4\int_{\de M}T_gv\,ds_g.
\end{align*}
Remember that, by conformal invariance,
\begin{equation}\label{eq:camconf}
	\begin{split}
	P^4_{\tilde g}(\varphi)&=e^{-4u}P^4_g(\varphi),\quad dV_{\tilde g} = e^{4u}dV_g \\ P^3_{\tilde g}(\varphi)&=e^{-3u}P^3_g(\varphi),\quad ds_{\tilde g} = e^{3u}ds_g.
	\end{split}
\end{equation}
These imply that
\begin{equation}\label{eq:invQ}
Q_g(v,v) = \int_{M} (e^{4u}P^4_{\tilde g}v)v\,e^{-4u}dV_{\tilde g}+2\int_{\de M}(e^{3u}P^3_{\tilde g}v)v\,e^{-3u}ds_{\tilde g} = Q_{\tilde g}(v,v).
\end{equation}
Now, multiply \eqref{eq:intro-QT} by $v$, integrate and use the conformal identities \eqref{eq:camconf} to obtain
\begin{align}
	\label{eq:inv1}\int_{M}(P^4_gu)v\,dV_g+2\int_{M}Q_gv\,dV_g &= 2\int_{M}Q_{\tilde g}e^{4u}v\,dV_g = 2\int_{M}Q_{\tilde g}v\,dV_{\tilde g}, \\ \label{eq:inv2}\int_{\de M}(P^3_gu)v\,ds_g+\int_{\de M}T_gv\,ds_g &= \int_{\de M}T_{\tilde g}e^{3u}v\,ds_g= \int_{\de M}T_{\tilde g}v\,ds_{\tilde g}.
\end{align}
Combining \eqref{eq:inv1} and \eqref{eq:inv2} we get
\begin{equation}\label{eq:invR}
\begin{split}
&2Q_g(u,v)+4\int_{M} Q_g v\,dV_g+4\int_{\de M}T_g v\,ds_g \\ &=2\left(\int_{M}(P^4_gu)v\,dV_g+2\int_{M}Q_gv\,dV_g+2\int_{\de M}(P^3_gu)v\,ds_g+2\int_{\de M}T_gv\,ds_g\right) \\ &=4\left(\int_{M}Q_{\tilde g}v\,dV_{\tilde g}+\int_{\de M}T_{\tilde g}v\,ds_{\tilde g}\right).
\end{split}
\end{equation}
Finally, by \eqref{eq:invQ} and \eqref{eq:invR}, we reach the desired conclusion:
$$S_g(u+v) = S_g(u)+S_{\tilde g}(v).$$
\end{proof}
In view of definitions \ref{def:quad} and \ref{def:S}, the identity \eqref{Quad} and the fact that $Q_g=3$ and $T_g=0$ in $(\Sph^4_+,g)$, we can rewrite the energy functional \eqref{eq:energy} as
\begin{equation}\label{fun:decomp}
I=S-\mathcal{N}_{Q}-\frac{4}{3}\mathcal{B}_{T},
\end{equation} with
\begin{equation*}
	S(u)=\int_{\Sphplus}(\Delta_{g}u)^{2}\,dV_{g}
	+2\int_{\Sphplus}|\grad u|^{2}\,dV_{g}
	+12\int_{\Sphplus}u\,dV_{g},
\end{equation*}
 and
\begin{equation}\label{eq:NQ}
	\mathcal{N}_{Q}(u)=\int_{\Sphplus}Qe^{4u}\,dV_{g},
	\qquad
	\mathcal{B}_{T}(u)=\int_{\Sph^{3}}Te^{3u}\,dS_{g}.
\end{equation}
\section{Conformal variations: proof of Theorem \ref{thm:KW}}\label{sec:KW}
Let $X$ be a $C^{1}$ vector field on $\Sphplus$ that preserves the boundary, namely
\begin{equation*}
\la X,\nu\ra=-\la X,x_5\ra=0 \quad \text{on }\Sph^{3}.
\end{equation*}
Let $\{\phi_{t}\}$ for $t\in(-\e,\e)$ be its flow:
\begin{equation*}
\phi_{0}=\Id,\qquad \frac{d}{dt}\phi_{t}(p)=X(\phi_{t}(p)).
\end{equation*}
We write
\begin{equation}\label{eq:Xonf}
X(f)=\frac{d}{dt}\Big|_{t=0}f(\phi_{t})
\end{equation}
for the action of $X$ on functions. Assume now that $\phi_{t}$ is conformal, that is, there exists a smooth function $P_{t}$ on $\Sphplus$ such that
\begin{equation}\label{eq:conf-flow}
\phi_{t}^{*}g=e^{2P_{t}}g,\qquad P_{0}=0.
\end{equation}
To simplify the notation, we will refer to the above metric as $g_t$. Given a solution $u$ to \eqref{eq:P}, we consider the smooth curve
\begin{equation}\label{eq:ut}
u_{t}=u\circ\phi_{t}+P_{t}.
\end{equation}

Since $u$ is a critical point of $I$, we have
\begin{equation}\label{eq:dI0}
\frac{d}{dt}\Big|_{t=0}I(u_{t})=0.
\end{equation}
\begin{proposition}\label{prop:Sinv}
For any boundary-preserving conformal flow \eqref{eq:conf-flow} and any $u\in\mathcal{H}$,
\begin{equation*}
\left.\frac{d}{dt}\right\vert_{t=0}S(u_{t})=0.
\end{equation*}
\end{proposition}
\begin{proof} First, we observe that $u_t \in \mathcal H$:
since $\phi_t$ is a $C^\infty$ diffeomorphism of $\Sphplus$, 
$u\circ\phi_t\in H^{2}(\Sphplus)$. Moreover $P_t\in C^\infty(\Sphplus)\subset H^{2}(\Sphplus)$, so $u_t\in H^{2}(\Sphplus)$.

\medskip
It remains to check the boundary condition. Fix $p\in\Sph^3$. By the chain rule,
\begin{equation*}
	d(u\circ\phi_t)_p(\nu(p))=du_{\phi_t(p)}\big(d\phi_t|_p(\nu(p))\big).
\end{equation*}
Since $\phi_t(\Sph^3)=\Sph^3$, we have $d\phi_t(T_p\Sph^3)=T_{\phi_t(p)}\Sph^3$.
As $\phi_t$ is conformal, $d\phi_t$ preserves orthogonality, hence it maps the normal line
$N_p(\Sph^3)=\mathrm{span}\{\nu(p)\}$ onto $N_{\phi_t(p)}(\Sph^3)=\mathrm{span}\{\nu(\phi_t(p))\}$.
Therefore there exists $\lambda_t(p)\in\R$ such that
\begin{equation*}
	d\phi_t|_p(\nu(p))=\lambda_t(p)\,\nu(\phi_t(p)).
\end{equation*}
Note that $\lambda_t(p)\neq 0$ since $d\phi_t|_p$ is invertible. Moreover $\lambda_0(p)=1$
	because $\phi_0=\Id$, and by continuity $\lambda_t(p)>0$ for $|t|$ small.

Consequently,
\begin{equation*}
	\frac{\partial(u\circ\phi_t)}{\partial \nu}(p)=\lambda_t(p)\,\frac{\partial u}{\partial \nu}(\phi_t(p))=0,
\end{equation*}
because $u\in\mathcal H$ and $\phi_t(p)\in\Sph^3$.

\medskip
Finally, we observe that the conformal factor $P_t$ satisfies 
\begin{equation}\label{eq:Pt}
	\frac{\partial P_t}{\partial_\nu}=0\qquad\text{on }\partial M.
\end{equation}
Indeed, since $\phi_t$ is a diffeomorphism with $\phi_t(\Sph^3)=\Sph^3$, then $\phi_t:(\Sphplus
,g_t)\to (\Sphplus,g)$ is an isometry. In particular, the second fundamental form $\mathbb I_g$ is transported via $\phi_t$, so $(\Sphplus,g_t)$ also has totally geodesic boundary.

\medskip
It is known that, under the conformal change of metrics $g_t=e^{2P_{t}}g$, we have $$e^{-2P_t}\mathbb{I}_{g_t} =\mathbb{I}_{g} +2\frac{\partial P_t}{\partial \nu}g$$ on $\Sph^{3}$, which gives \eqref{eq:Pt}. 

\medskip
Hence $$\frac{\partial u_t}{\partial \nu}=\frac{\partial(u\circ\phi_t)}{\partial \nu}+\frac{\partial P_t}{\partial_\nu}=0\quad \text{on }\Sph^3,$$ 
showing that $u_t\in\mathcal H$.

\medskip
By lemma \ref{lemma:conformal} and the invariance of $S$ under diffeomorphisms,
\begin{equation*}
	\begin{split}
		S_g(u_t) &= S_g(u\circ \phi_t+P_t) = S_g(P_t)+S_{e^{2P_t}g(u\circ \phi_t)  }\\ &= S_g(P_t)+S_{\phi_t ^*g}(u\circ\phi_t)=S_g(P_t)+S_g(u).
	\end{split}
\end{equation*}
Hence, since $S_g(u)$ is independent of $t$,
\begin{equation}\label{eq:first}
	\left.\frac{d}{dt}\right\vert_{t=0}S_g(u_t) = \left.\frac{d}{dt}\right\vert_{t=0}S_g(P_t).
\end{equation}
Let us prove that  $\left.\frac{d}{dt}\right\vert_{t=0}S_g(P_t)=0$. First, we see that
\begin{equation}\label{eq:der1}
	\begin{split}
\left.\frac{d}{dt}\right\vert_{t=0}S_g(P_t)&=\left.\frac{d}{dt}\right\vert_{t=0}\left(\int_{\Sphplus}(\Delta_{g}P_t)^{2}\,dV_{g}
	+2\int_{\Sphplus}|\grad P_t|^{2}\,dV_{g}
	+12\int_{\Sphplus}P_t\,dV_{g}\right)\\&=12\int_{\Sphplus}\dot P_0\,dV_g,
	\end{split}
\end{equation}
since all the quadratic terms vanish because $P_0=0$. Now, we aim to differentiate the identity $\phi_t^*g=e^{2P_t}g$ on $t=0$. Remember that, for every $C^1$ vector fields $Y,Z$ on $\Sphplus$, we have
\begin{equation*}
	(\phi_t^{*}g)(Y,Z)
	=g\!\left((\phi_t)_{*}Y,(\phi_t)_{*}Z\right)\circ\phi_t,
\end{equation*}
that is, 
\begin{equation*}
(\phi_t^{*}g)_{p}(Y_p,Z_p)
=g_{\phi_t(p)}\!\left(d\phi_t(p)[Y_p],d\phi_t(p)[Z_p]\right).
\end{equation*}
Therefore,
\begin{equation*}
	\begin{split}
	\frac{d}{dt}\Big|_{t=0}(\phi_t^{*}g)(Y,Z)
	&= X\!\big(g(Y,Z)\big)
	+ g\!\left(\frac{d}{dt}\Big|_{t=0}(\phi_t)_{*}Y,\ Z\right)
	+ g\!\left(Y,\ \frac{d}{dt}\Big|_{t=0}(\phi_t)_{*}Z\right)\\
	&= X\!\big(g(Y,Z)\big)+g([X,Y],Z)+g(Y,[X,Z])\\[5pt]&= (\mathcal{L}_X g)(Y,Z).
	\end{split}
\end{equation*} 
where the right-hand side is the Lie derivative. Consequently, we have
\begin{equation}\label{eq:lie}
\mathcal L_Xg=2\dot P_0e^{2P_0}g=2\dot P_0g.
\end{equation}
If we now take traces on \eqref{eq:lie} we get $2\text{div}_g X = 8 \dot P_0,$ which gives $$\dot P_0=\frac14\text{div}_gX.$$
Integrating on $\Sphplus$ and applying the divergence theorem:
\begin{equation}\label{eq:intP0}
\int_{\Sphplus}\dot P_0\,dV_g=\frac14\int_{\Sphplus}\text{div}_gX\,dV_g=\frac14\int_{\Sph^3}X\cdot \nu\,ds_g = 0,
\end{equation}
since $X$ preserves the boundary. The result follows from combining \eqref{eq:intP0}, \eqref{eq:der1} and \eqref{eq:first}.
\end{proof}

We now compute the derivatives of the nonlinear terms $\mathcal N_Q$ and $\mathcal B_T$ defined in \eqref{eq:NQ} along \eqref{eq:ut}. For the interior term,
\begin{equation*}
\frac{d}{dt}\Big|_{t=0}\int_{\Sphplus}Qe^{4u_{t}}\,dV_{g}
=\frac{d}{dt}\Big|_{t=0}\int_{\Sphplus}Q(x)e^{4u(\phi_{t}(x))}e^{4P_{t}(x)}\,dV_{g}(x).
\end{equation*}
Performing the change of variables $y=\phi_{t}(x)$, so that $x=\phi_{-t}(y)$, and using the conformal relation \eqref{eq:conf-flow} to relate the Jacobian to $P_{t}$, we obtain
\begin{equation}\label{eq:der-NQ-change}
\frac{d}{dt}\Big|_{t=0}\int_{\Sphplus}Qe^{4u_{t}}\,dV_{g}
=\frac{d}{dt}\Big|_{t=0}\int_{\Sphplus}Q(\phi_{-t}(y))e^{4u(y)}\,dV_{g}(y)
=-\int_{\Sphplus}X(Q)e^{4u}\,dV_{g}.
\end{equation}
Here we have used the group property of the flow, which implies that $\phi_{-t}= \phi^{-1}_{t}$, and \eqref{eq:Xonf}. More precisely, $$\frac{d}{dt}\Big|_{t=0}Q\circ\phi^{-1}_{t}=\frac{d}{dt}\Big|_{t=0}Q\circ\phi_{-t}=-\frac{d}{ds}\Big|_{s=0}Q\circ\phi_{s}=-X(Q).$$
Analogously, for the boundary term we obtain
\begin{equation}\label{eq:der-BT}
\frac{d}{dt}\Big|_{t=0}\int_{\Sph^{3}}Te^{3u_{t}}\,dS_{g}
=\frac{d}{dt}\Big|_{t=0}\int_{\Sph^{3}}T(\phi_{-t}(y))e^{3u(y)}\,dS_{g}(y)
=-\int_{\Sph^{3}}X(T)e^{3u}\,dS_{g}.
\end{equation}

Recalling the splitting of the functional given in \eqref{fun:decomp} and combining \eqref{eq:dI0} with Proposition \ref{prop:Sinv}, \eqref{eq:der-NQ-change} and \eqref{eq:der-BT}, we obtain the Kazdan--Warner type identity from  Theorem \ref{thm:KW}.

\section{Nonexistence criterions}\label{sec:nonexist}
As a direct consequence of \eqref{eq:KW}, we obtain the following obstruction to the existence of solutions to the prescribed curvature problem \eqref{eq:P}.

\begin{corollary}\label{cor:nonexist}
Assume there exists a boundary-preserving conformal vector field $X$ on $\Sphplus$ such that
\begin{equation*}
X(Q)\ge 0 \text{ on }\Sphplus,\qquad X(T)\ge 0 \text{ on }\Sph^{3},
\end{equation*}
and at least one of the two functions $X(Q)$, $X(T)$ is not identically zero. Then \eqref{eq:P} has no solutions in $\mathcal{H}$.
\end{corollary}
As a particular case, we obtain the classical Kazdan--Warner result, which establishes that there can be no solution for curvatures with the same monotonicity with respect to the same Euclidean direction.
\begin{corollary}\label{cor:coord}
	For $i\in\{1,\dots,4\}$ let $x_i:\Sphplus\to\R$ be the $i$-th Euclidean coordinate in $\R^{5}$ and set
	\[
	X_i=\grad_g x_i .
	\]
	Then $X_i$ is a boundary-preserving conformal vector field on $\Sphplus$ and any solution $u\in\mathcal H$
to  \eqref{eq:P} satisfies
	\[
	\int_{\Sphplus} \langle\grad_gx_i,\grad_gQ\rangle\,e^{4u}\,dV_g+\frac{4}{3}\int_{\Sph^{3}} \langle \nabla_gx_i,\nabla_gT\rangle\,e^{3u}\,ds_g=0.
	\]
	In particular, if for some $i$ one has
	\[
\langle\grad_gx_i,\grad_gQ\rangle \ge 0 \ \text{in }\Sphplus,\qquad \langle \nabla_gx_i,\nabla_gT\rangle\ge 0 \ \text{on }\Sph^{3},
	\]
	and at least one of the two functions above is not identically zero, then \eqref{eq:P}
	admits no solutions in $\mathcal H$.
\end{corollary}

	On $\Sph^{4}\subset\R^{5}$ one has the explicit formula
	\[
	X_i(x)=e_i-x_i x,
	\]
	hence $X_i(x_i)=1-x_i^{2}$ and therefore, for curvature functions $Q(x)=f(x_i)$ and $T(y)=g(y_i)$,
	\[
	X_i(Q)=f'(x_i)(1-x_i^{2}),\qquad X_i(T)=g'(y_i)(1-y_i^{2}).
	\]
	In particular, if $f'\ge0$ and $g'\ge0$ and at least one is not identically zero, then	\eqref{eq:P} has no solutions in $\mathcal H$.
	
	\medskip
	The preceding corollary can be extended to any conformal direction in the following sense:
\begin{corollary}\label{cor:conjugation}
	Let $\Psi:\Sphplus\to\Sphplus$ be a conformal diffeomorphism such that $\Psi(\Sph^3)=\Sph^3$.
	Fix $i\in\{1,\dots,4\}$ and set $X_i=\grad_g x_i$ on $\Sphplus$.
	Assume that
	\begin{equation}\label{eq:monotone-conj}
		X_i(Q\circ\Psi)\ge 0 \ \text{in }\Sphplus,
		\qquad
		X_i(T\circ\Psi)\ge 0 \ \text{on }\Sph^3,
	\end{equation}
	and that at least one of the two functions in \eqref{eq:monotone-conj} is not identically zero.
	Then the boundary problem \eqref{eq:P} admits no solutions in $\mathcal H$.
\end{corollary}

\begin{proof}
	Since $\Psi$ is conformal and preserves the boundary, the vector field $X=\Psi_*X_i$ is also boundary-preserving conformal vector field
	on $\Sphplus$. Moreover, for any smooth function $F$ on $\Sphplus$ one has
	\[
	X(F)\circ\Psi = X_i(F\circ\Psi),
	\]
	hence $X(Q)\ge0$ in $\Sphplus$ and $X(T)\ge0$ on $\Sph^3$ are equivalent to \eqref{eq:monotone-conj}.
	Therefore, the claim follows from Corollary \ref{cor:nonexist} applied to $X$.
\end{proof}

	We can construct easy explicit examples for corollary \ref{cor:conjugation} in the following way: given two distinct points $p,q\in\Sph^3$, one may choose $\Psi$ mapping $\pm e_i\in\Sph^3$
	to $p,q$, so that $X=\Psi_*X_i$ generates the boundary Möbius flow with fixed points $p$ and $q$.
	Then \eqref{eq:monotone-conj} is equivalent to say that $Q$ and $T$ are nondecreasing along that flow. 
	
	\medskip
	\noindent Let us illustrate this procedure.
	
	\medskip\noindent
	Identify $\Sph^4_+\subset \R^5$ as
	\begin{equation*}
	\begin{split}
	\Sph^4_+&=\{(\bar x,x_5)\in\R^4\times\R:\ |\bar x|^2+{x_5}^2=1,\ x_5\ge 0\},\\
	 \partial \Sph^4_+&=\{(\bar x,0)\in \R^4\times \{0\}: |\bar x|=1\}\cong \Sph^3\subset\R^4.
	\end{split}
	\end{equation*}
	Consider the restriction to $\Sph^4_+$ of the stereographic projection from the south pole $(0,\dots,0,-1)$ ,
	\[
	\Pi(\bar x,x_5)=\frac{\bar x}{1+x_5}\in B^4,
	\qquad
	\Pi^{-1}(y)=\left(\frac{2y}{1+|y|^2},\,\frac{1-|y|^2}{1+|y|^2}\right),
	\]
	which maps $\Sph^4_+$ conformally onto the unit ball $B^4\subset\R^4$ and restricts to the identity on the boundary
	$\Sph^3$ (since $\Pi(\bar x,0)=\bar x$ for $|\bar x|=1$).
	For $a\in B^4$, we define the following Möbius automorphism of $B^4:$
	\begin{equation*}\label{eq:ball-auto}
		\Phi_a(y)=\frac{(1-|a|^2)\,y+(1+2a\cdot y+|y|^2)\,a}{1+2a\cdot y+|a|^2|y|^2},
		\qquad y\in B^4,
	\end{equation*}
	so that $\Phi_a(B^4)=B^4$ and $\Phi_a(\Sph^3)=\Sph^3$.
	Set $\Psi=\Pi^{-1}\circ \Phi_a\circ \Pi:\ \Sph^4_+\to \Sph^4_+,$	which is a conformal diffeomorphism with $\Psi(\Sph^3)=\Sph^3$. Now choose, for instance, 
	\[
	a=\Big(0,\frac12,0,0\Big)\in B^4.
	\]
	Since $\Pi\vert_{\Sph^3}=\Id$, on the boundary we have $\Psi\vert_{\Sph^3}=\Phi_a\vert_{\Sph^3}$, and a direct computation gives
	\[
	p=\Psi(e_1)=\Phi_a(e_1)=\Big(\frac35,\frac45,0,0\Big),
	\qquad
	q=\Psi(-e_1)=\Phi_a(-e_1)=\Big(-\frac35,\frac45,0,0\Big).
	\]
	
	\medskip\noindent
	Let $X_1=\grad_g x_1$ on $\Sph^4_+$ and define $X=\Psi_*X_1$, which generates a boundary Möbius flow with fixed points
	$p$ and $q$ on $\Sph^3$. Consider the data
	\[
	Q(x)=3+\varepsilon\,(x_1\circ \Psi^{-1})(x)\quad \text{in }\Sph^4_+,
	\qquad
	T\equiv 1\quad \text{on }\Sph^3,
	\]
	with $\varepsilon>0$. Then
	\[
	(Q\circ\Psi)(x)=3+\varepsilon x_1,\qquad (T\circ\Psi)\equiv 1,
	\]
	and therefore
	\[
	X_1(Q\circ\Psi)=\varepsilon\,X_1(x_1)=\varepsilon\,|\grad_g x_1|^2=\varepsilon(1-x_1^2)\ge 0,
	\qquad
	X_1(T\circ\Psi)=0.
	\]
	Since $X_1(Q\circ\Psi)\not\equiv 0$, the assumptions of Corollary~\ref{cor:conjugation} are satisfied, hence the boundary
	problem \eqref{eq:P} admits no solutions in $\mathcal H$ for this choice of $(Q,T)$.

\appendix

\section{Weak Formulation}\label{app:wf}
Let $\mathcal{H}$ be the functional space defined in Section \ref{sec:var} and fix $u\in \mathcal H$. For any $v\in\mathcal{H}$, multiplying the first equation in \eqref{eq:P} by $v$ and integrating over $\Sphplus$ gives

\begin{equation}\label{eq:weak-start}
	\int_{\Sphplus}\left(\Delta_{g}^{2}u-2\Delta_{g}u+6\right)v\,dV_{g}=\int_{\Sphplus}2Qe^{4u}v\,dV_{g}.
\end{equation}
We integrate by parts in the term with $\Delta_{g}^{2}u$:
\begin{equation*}
	\begin{split}
		\int_{\Sphplus}\Delta_{g}^{2}u\,v\,dV_{g}
		&=\int_{\Sphplus}\Delta_{g}u\,\Delta_{g}v\,dV_{g}
		-\int_{\Sph^{3}}\Delta_{g}u\,\p_{\nu}v\,dS_{g}
		+\int_{\Sph^{3}}\p_{\nu}(\Delta_{g}u)\,v\,dS_{g}\\
		&=\int_{\Sphplus}\Delta_{g}u\,\Delta_{g}v\,dV_{g}
		+\int_{\Sph^{3}}\p_{\nu}(\Delta_{g}u)\,v\,dS_{g},
	\end{split}
\end{equation*}
where we used $\p_{\nu}v=0$ on $\Sph^{3}$. Using the boundary condition $-\p_{\nu}\Delta_{g}u=2Te^{3u}$, we obtain
\begin{equation}\label{eq:ibp-bilap-final}
	\int_{\Sphplus}\Delta_{g}^{2}u\,v\,dV_{g}
	=\int_{\Sphplus}\Delta_{g}u\,\Delta_{g}v\,dV_{g}
	-\int_{\Sph^{3}}2Te^{3u}v\,dS_{g}.
\end{equation}
Similarly,
\begin{equation}\label{eq:ibp-lap}
	-2\int_{\Sphplus}\Delta_{g}u\,v\,dV_{g}
	=2\int_{\Sphplus}\la \grad u,\grad v\ra\,dV_{g}
	-2\int_{\Sph^{3}}\p_{\nu}u\,v\,dS_{g}
	=2\int_{\Sphplus}\la \grad u,\grad v\ra\,dV_{g}.
\end{equation}
Plugging \eqref{eq:ibp-bilap-final} and \eqref{eq:ibp-lap} into \eqref{eq:weak-start} yields the weak formulation
\begin{equation*}
	\int_{\Sphplus}\Delta_{g}u\,\Delta_{g}v\,dV_{g}
	+2\int_{\Sphplus}\la \grad u,\grad v\ra\,dV_{g}
	+6\int_{\Sphplus}v\,dV_{g}
	-\int_{\Sph^{3}}2Te^{3u}v\,dS_{g}
	-\int_{\Sphplus}2Qe^{4u}v\,dV_{g}=0.
\end{equation*}

\section*{Acknowledgements}
Part of this work was carried out during the conference {\it PDEs at Grand Paradis
V}. The authors thank the organizers for providing a stimulating research environment.

A. DlT. acknowledges financial support from the Spanish Ministry of Science and Innovation (MICINN), through the
IMAG-Maria de Maeztu Excellence Grant CEX2020-001105-M/AEI/10.13039/501100011033. She is also supported by the FEDER-MINECO Grant PID2023-150166NB-I00;  RED2022-134784-T, funded
by MCIN/AEI/10.13039/501100011033; Fondi Ateneo - Sapienza Università di Roma;  DFG - Projektnummer: 561401741; and INdAM-GNAMPA  Projects 2024 with codice
CUP E53C23001670001 and 2025 with codice CUP E5324001950001.
\smallskip

Both authors have been supported by the MICIN/AEI through the Grant PID2024-155314NB-I00 and  by J. Andalucia (FQM-116).


\begin{thebibliography}{99}
	
\bibitem{Branson}
T.~P. Branson,
\textit{Differential operators canonically associated to a conformal structure},
Math. Scand. \textbf{57} (1985), no.~2, 293--345.

\bibitem{brendle2003}
S. Brendle,
\textit{Global existence and convergence for a higher order flow in conformal geometry,}
\textit{Annals of Math.} \textbf{158} (2003), 323-343.

\bibitem{ChangQingI}
S.-Y.~A. Chang and J.~Qing,
\textit{The zeta functional determinants on manifolds with boundary. I. The formula},
J. Funct. Anal. \textbf{147} (1997), 327--362.

\bibitem{ChangQingII}
S.-Y.~A. Chang and J.~Qing,
\textit{The zeta functional determinants on manifolds with boundary. II. Extremal metrics and compactness
	of isospectral set},
J. Funct. Anal. \textbf{147} (1997), 363--399.

\bibitem{ChangYangAnnals95}
S.-Y.~A. Chang and P.~C. Yang,
\textit{Extremal metrics of zeta function determinants on $4$-manifolds},
Ann. of Math. (2) \textbf{142} (1995), 171--212.

\bibitem{ChangYangS2}
S.-Y.~A. Chang and P.~C. Yang,
\textit{Prescribing Gaussian curvature on $\mathbb S^{2}$},
Acta Math. \textbf{159} (1987), 215--259.

\bibitem{CruzDeLaTorrePreprint}
S.~Cruz-Bl\'azquez and A.~DeLaTorre,
\textit{Conformal metrics on the four-dimensional half sphere with symmetric $Q$ and $T$ curvatures},
preprint.

\bibitem{CBRuiz} S. Cruz-Blázquez and D. Ruiz. \textit{Prescribing Gaussian and geodesic curvature on disks.} Adv. Non-linear Stud., 18(3):453–468, 2018.

\bibitem{djadli-malchiodi-Q1}
Z. Djadli and A. Malchiodi,
\textit{Existence of conformal metrics with constant Q-curvature},
\textit{Ann. of Math. (2)} \textbf{168} (2008), no. 3, 813–858.

\bibitem{Hamza}
H.~Hamza,
\textit{Sur les transformations conformes des vari\'et\'es riemanniennes \`a bord},
J. Funct. Anal. \textbf{92} (1990), no.~2, 403--447.

\bibitem{KazdanWarner74}
J.~L. Kazdan and F.~W. Warner,
\textit{Curvature functions for compact $2$-manifolds},
Ann. of Math. (2) \textbf{99} (1974), 14--47.

\bibitem{KazdanWarner75}
J.~L. Kazdan and F.~W. Warner,
\textit{Existence and conformal deformation of metrics with prescribed Gaussian and scalar curvatures},
Ann. of Math. (2) \textbf{101} (1975), 317--331.

\bibitem{LiLiLiu}
J. Li, Y. Li and P. Liu 
\textit	{The curvature on a $4$-dimensional Riemannian manifold  with$\int_M Q\,d V_g=8\pi^2$}
\textit {	Advances in Mathematics}
\textbf{ 231}, Issues 3–4, October–November 2012, 2194-2223.

\bibitem{LSRRS}
R. López-Soriano, F.J. Reyes-Sánchez and D. Ruiz,
\textit{Conformal Metrics on the Disk with Prescribed Negative Gaussian Curvature and Boundary Geodesic Curvature}, Preprint arxiv.org/abs/2602.15471.

\bibitem{malchiodi-struwe2006}
A. Malchiodi and M. Struwe,
\textit{Q-curvature flow on S4},
\textit{J. Differential Geom.} \textbf{73} (2006), no. 1, 1–44.

\bibitem{ndiaye2008}
C.B. Ndiaye,
\textit{Conformal metrics with constant Q-curvature for manifolds with boundary},
\textit{Commun. Anal. Geom.} \textbf{16} (2008), no. 5, 1049–1124.

\bibitem{ndiaye2009}
C.B. Ndiaye,
\textit{Constant T-Curvature conformal metrics on 4-manifolds with boundary},
\textit{Pacific J. Math.} \textbf{240} (2009), 151-184.

\bibitem{ndiaye2011}
C.B. Ndiaye,
\textit{Q-Curvature flow on 4 manifolds with boundary},
\textit{Math. Z.} \textbf{269} (2011), 83-114.

\bibitem{Paneitz}
S.~Paneitz,
\textit{A quartic conformally covariant differential operator for arbitrary pseudo-Riemannian manifolds}
(summary),
SIGMA \textbf{4} (2008), 036.

\bibitem{ruiz24} D. Ruiz. \textit{Conformal metrics of the disk with prescribed Gaussian and geodesic curvatures}. Math. Ann., 390(1):1049–1075, 2024.

\bibitem{struwe} M. Struwe, \textit{The prescribed curvature flow on the disc},
Ars Inveniendi Analytica (2024), Paper No. 5, 59 pp. 

\bibitem{wei-xu1998}
J. Wei and X. Xu,
\textit{On conformal deformations of metrics on $S^n$},
\textit{J. Funct. Anal.} \textbf{157} (1998), no. 1, 292-325.

\end{thebibliography}
\end{document}